\documentclass[12pt,a4paper]{article}
\usepackage{amsfonts}
\usepackage{eurosym}
\usepackage{latexsym}
\usepackage{amsmath}
\usepackage{amssymb}
\usepackage[margin=.7in]{geometry}
\usepackage[english]{babel}
\usepackage{scalerel,stackengine}
\usepackage{extpfeil}
\usepackage{graphicx}
\usepackage{ntheorem}
\usepackage{tikz-cd}
\usepackage{MnSymbol}
\usepackage{stackengine}

\theoremseparator{.}

\stackMath
\newcommand\reallywidehat[1]{\savestack{\tmpbox}{\stretchto{  \scaleto{    \scalerel*[\widthof{\ensuremath{#1}}]{\kern-.6pt\bigwedge\kern-.6pt}    {\rule[-\textheight/2]{1ex}{\textheight}}  }{\textheight}}{0.5ex}}\stackon[1pt]{#1}{\tmpbox}}
\newtheorem{theorem}{Theorem}

\newtheorem{corollary}[theorem]{Corollary}

\newtheorem{lemma}[theorem]{Lemma}

\newtheorem{proposition}[theorem]{Proposition}
\newtheorem{remark}[theorem]{Remark}

\def\div{\operatorname{div}}

\def\sgn{\operatorname{sgn}}

\begin{document}
\makeatletter
\def\blfootnote{\xdef\@thefnmark{}\@footnotetext}
\makeatother

\date{}
\title{\textbf{Bourgain-Brezis spaces obtained by real interpolation}}
\author{ Eduard Curc\u a \thanks{%
Faculty of Mathematics, Informatics and Mechanics, University of Warsaw,
Banacha 2, Warsaw, 02-097, Poland} }
\maketitle

\begin{abstract}
\setlength{\parindent}{2cm}In 2002, Bourgain and Brezis proved that for the
space $X=W^{1,d}$ (on $\mathbb{T}^{d}$, with $d\geq2$) we have the equality
of images 
\begin{equation}
\div  (L^{\infty}\cap X)=\div  X,  \tag{$\ast$}
\end{equation}
i.e., given a vector field $v\in X$ there exists a vector field $u\in
L^{\infty }\cap X$ such that $\div  u=\div  v $.

In this paper we show that if $X$ is a function space satisfying ($\ast$)
then, any real interpolation space $X_{\theta,q}=(L^{\infty},X)_{\theta,q}$
(where $\theta\in (0,1)$ and $q\in [1,\infty)$) also satisfies ($\ast$). The
proof is based on a general method that allows us to interpolate solutions
of linear equations.
\end{abstract}

\makeatletter

\makeatother

\makeatletter

\makeatother

\makeatletter

\makeatother

\bigskip

\section{Introduction}

\blfootnote{Keywords: Divergence equation, Real interpolation, Sobolev spaces.}
\blfootnote{MSC 2020 classification: 35F05, 46B70, 46E35.}

Let $X$ be a Banach function space on $\mathbb{T}^{d}$, where $d\geq 2$. We
say that $X$ is a Bourgain-Brezis space ($BB$ space, for short) if for any
vector field $v\in X$ there exists another vector field $u\in L^{\infty
}\cap X$ such that we have 
\begin{equation}
\div u=\div v,  \label{div-00}
\end{equation}%
in the sense of distributions on $\mathbb{T}^{d}$. In short, we can express
this as the equality of images 
\begin{equation}
\div (L^{\infty }\cap X)=\div X.  \label{div-0}
\end{equation}

In other words, if $X$ is a $BB$ space, the divergence operator does not
make distinction between $X$ and the usually smaller space $L^{\infty }\cap
X $.

By using an involved construction Bourgain and Brezis (\cite{BBold}, 2002)
proved that the Sobolev space $X=W^{1,d}$ is a $BB$ space. They also provided a
much simpler argument (based on duality; see \cite[Section 4]{BBold}) for
the fact that $X=H^{1}=W^{1,2}$ is a $BB$ space, when $d=2$. Later, using a
simple duality argument, Mazya (\cite{Maz}, 2007) has shown that the Sobolev
space $X=H^{d/2}$ is a $BB$ space, for any $d\geq 2$. These simple arguments%
\footnote{%
There are also other places in literature where arguments similar to the one
of Mazya or the one of Bourgain and Brezis have been used; see for instance 
\cite[Remark 34]{CE-BB} and the references therein.} are, however, based on
the spectral structure of the spaces $H^{d/2}$ and do not seem to be
appropriate in the case of other types of spaces.

On the other hand, the constructive arguments of Bourgain and Brezis (see 
\cite{BBold} and \cite{BB}) have been generalized to other critical function
spaces like the Triebel-Lizorkin scale $F_{q}^{d/p,p}$ (see \cite{BMR} and 
\cite{BRWY}).

It remains the question whether we can find $BB$ spaces that are
\textquotedblleft not like $H^{d/2}$\textquotedblright\ (or spectral) in an
easy way, not by a repeating or modifying the constructions provided in \cite%
{BBold} and \cite{BB}. One attempt in this direction was made in \cite{CE-BB}
where we \textquotedblleft interpolated\textquotedblright\ the non-trivial
situation of the space $H^{d/2}$ and the trivial situation of the Besov
spaces $B_{1}^{d/p,p}$ (note that $B_{1}^{d/p,p}\hookrightarrow L^{\infty }$
and hence, $B_{1}^{d/p,p}$ is trivially a $BB$ space). Simplifying the
statement (and adapting it to the case of the torus), we have obtained (see 
\cite{CE-BB}):

\begin{theorem}
\label{DIV} If $p\in \lbrack 2,\infty )$ and $q\in (1,2)$, then 
\begin{equation}
\div  (L^{\infty }\cap B_{2}^{d/p,p})\hookleftarrow \div 
(B_{q}^{d/p,p}).  \label{div-DIV}
\end{equation}

In other words, given a vector field $v\in B_{q}^{d/p,p}$ there exists
another vector field $u\in L^{\infty }\cap B_{2}^{d/p,p}$ such that (\ref%
{div-00}) holds.
\end{theorem}

The proof of Theorem \ref{DIV} is based on a variant of the complex
interpolation method (the $\mathcal{W}$-method; see \cite[Section 3]{CE-BB}
). For technical reasons (see \cite[Section 5.2]{CE-BB}) we were forced to
give up some regularity in the \textquotedblleft third
parameter\textquotedblright : while the source is in $B_{q}^{d/p,p}$, with $%
q\in (1,2)$, the Besov regularity of the solution is given by the larger
space $B_{2}^{d/p,p}$. Since we do not have the same function space on both
sides of (\ref{div-DIV}), we can not speak here about a $BB$ space being involved.

In this paper we use a similar approach that is adapted to the real
interpolation method instead of the complex one. Our main result is the
following:

\begin{theorem}
\label{th.BB}Let $X$ be a $BB$ space on $\mathbb{T}^{d}$ such that $\left(
L^{\infty },X\right) $ is an interpolation couple, $L^{\infty }\cap X$ has a
separable predual and $X\hookrightarrow H^{-m}$, for some $m>d/2$. Then, for
any $\theta \in (0,1)$, $q\in \lbrack 1,\infty )$ the real interpolation
space $X_{\theta ,q}:=\left( L^{\infty },X\right) _{\theta ,q}$ is a $BB$
space.
\end{theorem}

Theorem \ref{th.BB} gives us a method for constructing new $BB$ spaces from
old. In particular, combining Theorem \ref{th.BB} with the result of Mazya
that $H^{d/2}$ is a $BB$ space, we get the following result.

\begin{corollary}
\label{cor.BB}For any $\theta \in (0,1)$, $q\in \lbrack 1,\infty )$ the real
interpolation space $X_{\theta ,q}:=\left( L^{\infty },H^{d/2}\right)
_{\theta ,q}$ is a $BB$ space.
\end{corollary}

Note that there is no loss of regularity here: the source belongs to $\left(
L^{\infty },H^{d/2}\right) _{\theta ,q}$ and the solution belongs to the
smaller space $L^{\infty }\cap \left( L^{\infty },H^{d/2}\right) _{\theta
,q} $. The result is also non-trivial in the sense that (see Section \ref%
{s.f.remarks2}) 
\begin{equation}
\left( L^{\infty },H^{d/2}\right) _{\theta ,q}\not\hookrightarrow L^{\infty
}.  \label{nontriv}
\end{equation}

However, while Corollary \ref{cor.BB} can be easily obtained by our methods
in this paper, it has one major drawback. Namely, the fact that we do not
have at this moment an easy way to describe the interpolation spaces like $%
\left( L^{\infty },H^{d/2}\right) _{\theta ,q}$ (other than the definition
provided, for instance, via the $K-$method). Concerning this, we have proved
in \cite[Theorem 1.4]{CE-1} (see also \cite[Theorem 2]{CE-S}) that the space 
$(L^{\infty },H^{d/2})_{\theta ,q}$ contains no Triebel-Lizorkin space $%
F_{\tau }^{\sigma ,p}$ where $\tau \in \lbrack 1,\infty ]$, $\sigma =\theta
d/2$ and $1/p=\theta /2$. If, independently to the result of Corollary \ref%
{cor.BB}, we restrict the attention to describing the pathological space $%
\left( L^{\infty },H^{d/2}\right) _{\theta ,q}$, one may ask whether such a
space appears naturally in practice. The present paper represents such an
instance\footnote{One can even say that the purpose of this paper is to draw attention to the
\textquotedblleft explicit computation\textquotedblright\ of the spaces like 
$(L^{\infty },H^{d/2})_{\theta ,q}$. See Section \ref{s.f.remarks2} and
Corollary \ref{cor.BB'}.}. Proving that a given function space is a $BB$
space is in general a non-trivial task. In contrast, the proof of Theorem \ref%
{th.BB} (and Corollary \ref{cor.BB}) is relatively cheap and natural.

\bigskip Theorem \ref{th.BB} is a direct consequence of the following
general functional analytic fact:

\begin{theorem}
\label{th.BB.g}Let $X_{0},X_{1},E,F$ be some non-trivial Banach spaces such
that $X_{0},X_{1}\hookrightarrow E$, the spaces $X_{0}$, $X_{0}\cap X_{1}$
have separable preduals and $E$ is separable and reflexive. Suppose that $%
T:E\rightarrow F$ is a bounded linear operator satisfying the equality 
\begin{equation*}
T(X_{0}\cap X_{1})=T(X_{1}).
\end{equation*}

Then, for any $\theta \in (0,1)$, $q\in \lbrack 1,\infty )$ we have 
\begin{equation*}
T(X_{0}\cap X_{\theta ,q})=T(X_{\theta ,q}),
\end{equation*}%
where $X_{\theta ,q}:=(X_{0},X_{1})_{\theta ,q}$.
\end{theorem}

Theorem \ref{th.BB} is obtained directly from Theorem \ref{th.BB.g} by
setting $X_{0}=L^{\infty }$ and $X_{1}=X$. In turn, Theorem \ref{th.BB.g}
follows directly from the more general Theorem \ref{th.int} below. The
strategy used in the proof of Theorem \ref{th.int} is essentially the same
as the one used in the proof of Lemma 29 in \cite{CE-BB}, and we are using
here facts that were proved in \cite[Section 3.3]{CE-BB}. However, there are
some adaptations to be made; in particular, the proof of Theorem \ref{th.int}
is slightly simpler than its analogue, Lemma 29 in \cite{CE-BB}. As in \cite[%
Section 3.3]{CE-BB} we construct solutions by \textquotedblleft gluing
together\textquotedblright\ the solutions that we get on the two components of the boundary of
the standard strip. This is possible thanks to the fact proved in \cite[%
Section 3]{Lind} that the real interpolation method admits a formulation
that resembles the one of the complex interpolation method. Roughly speaking, the real
method is a complex method seen on the \textquotedblleft other
side\textquotedblright\ of the Fourier transform. In this view, we implement
the strategy of \cite[Section 3.3]{CE-BB} on the \textquotedblleft other
side\textquotedblright\ of the Fourier transform, getting rid in this way of
some problems related to the finer genuine complex interpolation variant.

\bigskip

\noindent \textbf{Notation and conventions.} Throughout the paper we use
mainly standard notation concerning the Sobolev, Triebel-Lizorkin and Besov
spaces on $\mathbb{T}^{d}$ as in \cite{S-T} (see also Section 2 of \cite{BMR}%
). For the notation and the methods related to standard interpolation theory
we refer to \cite{LB}.

We use the symbols $\lesssim $ and $\sim $ as follows. For two non-negative
variable quantities $A$ and $B$ we write $A\lesssim B$ if there exists a
constant $C>0$ such that $A\leq CB$. If $A\lesssim B$ and $B\lesssim A$,
then we write $A\sim B$.

When $X$ is a function space on $\mathbb{T}^{d}$ and $u=(u_{1},...,u_{d})$
is a vector field on $\mathbb{T}^{d}$ where each $u_{j}$ belongs to $X$, we
write $u\in X$ instead of $u\in X\times ...\times X=X^{d}$. Also, we denote
by $X_{\sharp }$ the space of those distributions $f\in X$, for which $%
\widehat{f}(0)=0$. In this paper all the function spaces will be considered
on $\mathbb{T}^{d}$ (even when we quote results that where initially
obtained in the case of $\mathbb{R}^{d}$, since transferring these results
from $\mathbb{R}^{d}$ to $\mathbb{T}^{d}$ is an easy task in general).
\medskip

\textit{Embeddings of images. }Suppose $A$, $B$, $F$ are Banach spaces and $%
T_{1}:A\rightarrow F$, $T_{2}:B\rightarrow F$ are two linear operators. In
what follows we write 
\begin{equation}
T_{2}(B)\hookrightarrow T_{1}(A),  \label{Images}
\end{equation}
if there exists some constant $C>0$ such that for any $b\in B$ there exists
some $a\in A$ such that $T_{1}a=T_{2}b$, in $F$, and $\lVert a\rVert
_{A}\leq C\lVert b\rVert _{B}$. We write 
\begin{equation*}
T_{2}(B)= T_{1}(A),
\end{equation*}
when we have both (\ref{Images}) and 
\begin{equation*}
T_{1}(A)\hookrightarrow T_{2}(B).
\end{equation*}

\textit{Weighted spaces.} Given $q\in[1,\infty)$, a Banach space $Z$ and a measurable function $\omega: \mathbb{R}\rightarrow Z$ we denote by $L^{q}(\mathbb{R},Z)$ the space of all strongly measurable functions $f:\mathbb{R}\rightarrow Z$ for which the norm 
\begin{equation*}
\Vert f \Vert_{L^{q}(\mathbb{R},Z)}:= \left( \int_{\mathbb{R}} \Vert f(t) \Vert_{Z} ^{q} \omega (t)dt \right)^{1/q},
\end{equation*}
is finite.
\medskip

Throughout the paper the dimension $d$ will be always assumed to be (an integer) at least $2$, unless otherwise is specified.  
\section*{Acknowledgements}

This work was supported by the National Science Centre, Poland, CEUS
programme, project no. 2020/02/Y ST1/00072.

\section{Real interpolation of solutions}

\subsection{Real interpolation as a complex interpolation}

\label{sec.C} In order to prove our interpolation result, it is convenient to
use a description of the real interpolation method via complex analysis. In order to present such a description\footnote{Besides \cite{Lind} there are several instances in the literature where analytic functions were used in the context of the real interpolation (see the discussion and the references in \cite{Lind}).}, we follow
with unessential changes the description provided  in \cite{Lind}. Let $S$ be the standard open
strip 
\begin{equation*}
S:=\{z\in \mathbb{C}\mid 0<\Re z<1\},
\end{equation*}%
and let $\overline{S}$ be its closure.

To a Banach space $A$ we associate the linear space $\mathcal{H}^{1}(%
\overline{S},A)$ of all continuous functions $f:\overline{S}\rightarrow A$
that are analytic in $S$ and such that 
\begin{equation*}
\sup_{x\in \lbrack 0,1]}\Vert f(x+i\cdot )\Vert _{L^{1}(\mathbb{R}%
,A)}<\infty .
\end{equation*}

Consider a compatible couple $\left( A_{0},A_{1}\right) $ of Banach spaces
and a parameter $q\in \lbrack 1,\infty ]$. Let $\widehat{\mathcal{F}}^{q}=%
\widehat{\mathcal{F}}^{q}(A_{0},A_{1})$ be the linear space of all functions 
$f\in \mathcal{H}^{1}(\overline{S},A_{0}+A_{1})$, such that $f(j+it)\in
A_{j} $ for any $j=0,1$ and any $t\in \mathbb{R}$, for which the quantity
\begin{equation}
\lVert f\rVert _{\widehat{\mathcal{F}}^{q}}:=\max_{j=0,1}\Vert \widehat{%
f(j+i\cdot )}\Vert _{L^{q}(\mathbb{R},A_{j})}.  \label{W1}
\end{equation}
is finite. One can see that this quantity defines a norm on   $\widehat{\mathcal{F}}^{q}$.

For $0<\theta <1$ we consider on $A_{0}\cap A_{1}$ the norm given by 
\begin{equation}
\lVert a\rVert _{\theta ,q}^{c}:=\inf \left\{ \lVert f\rVert _{\widehat{%
\mathcal{F}}^{q}}\mid a=f(\theta )\text{, }f\in \widehat{\mathcal{F}}%
^{q}\left( A_{0},A_{1}\right) \right\} \text{,}  \label{W-def}
\end{equation}%
for any $a\in A_{0}\cap A_{1}$, and we denote by $(A_{0},A_{1})_{\theta
,q}^{c}$ the space obtained by completing $A_{0}\cap A_{1}$ with respect to
this norm. By by considering analytic functions of the form $f(z):=\exp
((z-\theta )^{2})a$ one can easily see that $\lVert a\rVert _{\theta
,q}^{c}<\infty $, for any $a\in A_{0}\cap A_{1}$.

Note that, if $q<\infty $, the space $A_{0}\cap A_{1}$ is dense in the usual
real interpolation space $(A_{0},A_{1})_{\theta ,q}$ (see \cite[Theorem 3.4.2%
]{LB}). In this case, \cite[Theorem 1.1]{Lind} implies that 
\begin{equation}
(A_{0},A_{1})_{\theta ,q}=(A_{0},A_{1})_{\theta ,q}^{c},  \label{C-0}
\end{equation}%
with equivalence of norms.

\bigskip

In what follows it will be convenient to work with functions that are more
regular on the boundary of the standard strip and that are close to be
scalar. For this purpose we give the following lemma.

\begin{lemma}
\label{lem.dens}Let $\left( A_{0},A_{1}\right) $ be a compatible couple of
Banach spaces and fix a parameter $q\in \lbrack 1,\infty )$. Define the
subspace $\widehat{\mathcal{F}}_{0}^{q}(A_{0},A_{1})$ of $\widehat{\mathcal{F%
}}^{q}(A_{0},A_{1})$ to be the set of all elements of the form 
\begin{equation*}
\sum_{l=1}^{N}a_{n}f_{n},
\end{equation*}%
where $N\in \mathbb{N}^{\ast }$, each element $a_{n}$ belongs to $A_{0}\cap
A_{1}$ and each scalar function $f_{n}\in \mathcal{H}^{1}(\overline{S},%
\mathbb{C})$ is smooth and bounded on $\partial S$.

We have that $\widehat{\mathcal{F}}_{0}^{q}(A_{0},A_{1})$ is dense in $%
\widehat{\mathcal{F}}^{q}(A_{0},A_{1})$.
\end{lemma}

\noindent\textbf{Proof.} Let us first make some standard remarks. For some $%
\delta >0$ consider the function $g_{\delta }(z):=\exp (\delta z^{2})$. By a
direct computation we see that $g_{\delta }f\in \widehat{\mathcal{F}}%
^{q}(A_{0},A_{1})$, whenever $f\in \widehat{\mathcal{F}}^{q}(A_{0},A_{1})$,
and 
\begin{equation}
\left\Vert f-g_{\delta }f\right\Vert _{\widehat{\mathcal{F}}^{q}}\rightarrow
0,  \label{dens-0}
\end{equation}%
when $\delta \rightarrow 0$.

Also, fix a non-negative function $\varphi \in C_{c}^{\infty }(\mathbb{R)}$
of integral $1$, and for $\varepsilon >0$ denote by $\varphi _{\varepsilon }$
the function $\varepsilon ^{-1}\varphi (\varepsilon ^{-1}\cdot )$. For $f\in 
\widehat{\mathcal{F}}^{q}(A_{0},A_{1})$ the convolution $f\ast \varphi
_{\varepsilon }\in \widehat{\mathcal{F}}^{q}(A_{0},A_{1})$ is defined as in 
\cite{CE-BB}, i.e., 
\begin{equation*}
f\ast \varphi _{\varepsilon }(x+iy):=\int_{\mathbb{R}}f(x+i(y-t))\varphi
_{\varepsilon }(t)dt,
\end{equation*}%
for any $x\in \lbrack 0,1]$, $y\in \mathbb{R}$.

Note that 
\begin{equation}
\left\Vert f-f\ast \varphi _{\varepsilon }\right\Vert _{\widehat{\mathcal{F}}%
^{q}}=\max_{j=0,1}\Vert |\widehat{\varphi }(\varepsilon \cdot )-1|\widehat{%
f(j+i\cdot )}\Vert _{L^{q}(\mathbb{R},A_{j})}\rightarrow 0,  \label{dens-1}
\end{equation}%
when $\varepsilon \rightarrow 0$, by the dominated convergence theorem.
Hence, one can replace $f$ \ by a function of the form $f\ast \varphi
_{\varepsilon }$ which in particular, belongs to $\mathcal{F}(A_{0},A_{1})$
(see \cite[Section 4.1]{LB}). Thanks to Lemma 4.2.3 in \cite{LB} there
exists a sequence $(f_{N})_{N\geq 1}$ of functions in $\widehat{\mathcal{F}}%
_{0}^{q}(A_{0},A_{1})$ such that 
\begin{equation*}
\max_{j=0,1}\sup_{t\in \mathbb{R}}\left\Vert f(j+it)-f_{N}(j+it)\right\Vert
_{A_{j}}\rightarrow 0,
\end{equation*}
when $N\rightarrow \infty $. This shows that, for any $\delta >0$, we have 
\begin{equation*}
\max_{j=0,1}\left\Vert ((g_{\delta }f-g_{\delta }f_{N})(j+i\cdot ))^{\wedge
}(\xi )\right\Vert _{A_{j}}\lesssim \max_{j=0,1}\int_{\mathbb{R}}e^{-\delta
t^{2}}\left\Vert f(j+it)-f_{N}(j+it)\right\Vert _{A_{j}}dt\rightarrow 0,
\end{equation*}%
when $N\rightarrow \infty $, uniformly $\xi \in \mathbb{R}$. Hence, 
\begin{eqnarray}
\left\Vert (g_{\delta }f-g_{\delta }f_{N})\ast \varphi _{\varepsilon
}\right\Vert _{\widehat{\mathcal{F}}^{q}} &=&\max_{j=0,1}\Vert \widehat{%
\varphi }(\varepsilon \cdot )((g_{\delta }f-g_{\delta }f_{N})(j+i\cdot
))^{\wedge }\Vert _{L^{q}(\mathbb{R},A_{j})}  \notag \\
&\leq &\Vert \widehat{\varphi }(\varepsilon \cdot )\Vert _{L^{q}(\mathbb{R}%
)}\max_{j=0,1}\Vert ((g_{\delta }f-g_{\delta }f_{N})(j+i\cdot ))^{\wedge
}\Vert _{L^{\infty }(\mathbb{R},A_{j})}\rightarrow 0,  \label{dens-3}
\end{eqnarray}%
when $N\rightarrow \infty $, for any fixed $\varepsilon >0$.

Combining (\ref{dens-1}), (\ref{dens-3}) and the inequality%
\begin{equation*}
\left\Vert g_{\delta }f-(g_{\delta }f_{N})\ast \varphi _{\varepsilon
}\right\Vert _{\widehat{\mathcal{F}}^{q}}\leq \left\Vert g_{\delta
}f-(g_{\delta }f)\ast \varphi _{\varepsilon }\right\Vert _{\widehat{\mathcal{%
F}}^{q}}+\left\Vert (g_{\delta }f-g_{\delta }f_{N})\ast \varphi
_{\varepsilon }\right\Vert _{\widehat{\mathcal{F}}^{q}},
\end{equation*}%
we get that $g_{\delta }f$ can be approximated by elements in $\widehat{%
\mathcal{F}}_{0}^{q}(A_{0},A_{1})$. This and (\ref{dens-0}) concludes the
proof. \hfill $\square $

\bigskip

Consider the subspace $\widehat{\mathcal{F}}_{00}^{q}(A_{0},A_{1})$ of $%
\widehat{\mathcal{F}}_{0}^{q}(A_{0},A_{1})$ to be the set of all elements of
the form 
\begin{equation*}
\sum_{l=1}^{N}a_{n}f_{n},
\end{equation*}%
where $N\in \mathbb{N}^{\ast }$, each element $a_{n}$ belongs to $A_{0}\cap
A_{1}$, each function $f_{n}\in \mathcal{H}^{1}(\overline{S},\mathbb{C})$ is
such that $t\rightarrow f_{n}(j+it)$, $j=0,1$, are Schwartz on $\mathbb{R}$
and belong to $L^{2}(\exp \left( 4t ^{2}\right) ,\mathbb{C})$. As we already
noticed after (\ref{W-def}), any $a\in A_{0}\cap A_{1}$ can be written as $%
a=f(\theta )$ for some $f\in \widehat{\mathcal{F}}_{00}^{q}\left(
A_{0},A_{1}\right) $. On $A_{0}\cap A_{1}$ we introduce the norm

\begin{equation*}
\lVert a\rVert _{C_{\theta }^{q}}:=\inf \left\{ \lVert f\rVert _{\widehat{%
\mathcal{F}}^{q}}\mid a=f(\theta )\text{, }f\in \widehat{\mathcal{F}}%
_{00}^{q}\left( A_{0},A_{1}\right) \right\} \text{.}
\end{equation*}

Note that if $f\in \widehat{\mathcal{F}}_{0}^{q}(A_{0},A_{1})$ is such that $%
f(\theta )=a$, then $\widetilde{f}:=\exp (8(z-\theta )^{2})f\in \widehat{%
\mathcal{F}}_{00}^{q}(A_{0},A_{1})$ and $\widetilde{f}(\theta )=a$. This
observation together with Lemma \ref{lem.dens} leads to the following fact.

\begin{lemma}
\label{lem.d}The completion of $A_{0}\cap A_{1}$ with respect to the norm $%
\lVert \cdot \rVert _{C_{\theta }^{q}}$ is the space $(A_{0},A_{1})_{\theta
,q}^{c}$ (with equivalent norms).
\end{lemma}

\subsection{Boundary values of functions on the strip}

\label{sec.B}

Let $Z$ be a Banach space and consider a function $f\in C_{b}^{1}\left( 
\mathbb{R},Z\right) \cap L^{2}\left( \mathbb{R},Z\right) $. The Hilbert
transform of $f$ is defined by

\begin{equation}
Hf(t):=\frac{1}{\pi }\lim_{\varepsilon \rightarrow 0}\int_{\varepsilon
<|t-s|<1/\varepsilon }\frac{f(s)}{t-s}ds\text{,}  \label{Hilb}
\end{equation}%
for $t\in \mathbb{R}$. According to Lemma 25 in \cite{CE-BB}, for any such
function $f$ the limit in (\ref{Hilb}) exists point-wise and it defines a
function $Hf\in C_{b}\left( \mathbb{R},Z\right) $.

As in \cite[Section 3.3.1]{CE-BB} we introduce the operators $H_{j}^{S}$ and 
$R_{j}^{S}$ that will be used to describe the boundary behaviour of analytic
functions on the strip. We recall that, for each $j=0,1$, 
\begin{equation*}
H_{j}^{S},R_{j}^{S}:C_{b}^{1}(\mathbb{R},Z)\cap L^{2}(\left\langle
t\right\rangle ^{d},Z)\rightarrow C_{b}(\mathbb{R},Z),
\end{equation*}%
are defined by%
\begin{equation*}
H_{j}^{S}f(t):=-\frac{i}{2}f(t)-\frac{(-1)^{j}}{2}Hf(t),
\end{equation*}%
and 
\begin{equation*}
R_{j}^{S}f(t):=\rho _{j}\ast f(t)\text{, }
\end{equation*}%
respectively, where $\rho _{j}\in C_{b}(\mathbb{R},\mathbb{C)\cap }L^{2}(%
\mathbb{R},\mathbb{C)}$ are the functions 
\begin{equation*}
\rho _{j}(t):=\frac{(-1)^{j}}{2\pi i}\frac{1}{1-2j+it}.
\end{equation*}

The operator $H_{j}^{S}$ is well-defined thanks to the fact that the Hilbert
transform is well-defined. Also, note that for $f\in L^{2}(\left\langle
t\right\rangle ^{d},Z)$ the quantity $R_{j}^{S}f$ is well-defined and by the
dominated convergence theorem we have $R_{j}^{S}f\in C_{b}\left( \mathbb{R}%
,Z\right) $.

The Fourier symbols corresponding to the operators $H_{j}^{S}$ and $%
R_{j}^{S} $ are 
\begin{equation*}
\chi _{j}(\xi ):=-\frac{i}{2}-\frac{(-1)^{j}}{2}i\sgn (\xi ),
\end{equation*}%
and $\widehat{\rho }_{j}(\xi )$ (for $\xi \in \mathbb{R}$) respectively.
(Note that, since $\rho _{j}$\ are $L^{2}$ functions, $\widehat{\rho }_{j}$
are also $L^{2}$ functions.)

In other words we have 
\begin{equation*}
\widehat{H_{j}^{S}f}(\xi )=\chi _{j}(\xi )\widehat{f}(\xi ),
\end{equation*}%
and 
\begin{equation*}
\widehat{R_{j}^{S}f}(\xi )=\widehat{\rho }_{j}(\xi )\widehat{f}(\xi )\text{, 
}
\end{equation*}%
(for $\xi \in \mathbb{R}$) for any $f$ in $C_{b}^{1}(\mathbb{R},Z)\cap
L^{2}(\left\langle t\right\rangle ^{d},Z)$. Clearly, $\chi _{j}$ are bounded
functions. One can see that $\widehat{\rho }_{j}$ are also bounded as the
following easy lemma shows us.

\begin{lemma}
\label{lem.ro}We have $\widehat{\rho }_{j}\in L^{\infty }(\mathbb{R},\mathbb{%
C})$, for any $j=0,1$.
\end{lemma}

\noindent\textbf{Proof.} It suffices to see that the function $\rho \in
C_{b}(\mathbb{R},\mathbb{C)\cap }L^{2}(\mathbb{R},\mathbb{C)}$ with 
\begin{equation*}
\rho (t):=\frac{1}{1+it},
\end{equation*}%
has a bounded Fourier transform. We can write $\rho =g_{1}+g_{2}$, where 
\begin{equation*}
g_{1}(t):=\frac{1}{1+t^{2}}\text{ \ \ and \ }g_{2}(t):=-\frac{it}{1+t^{2}}.
\end{equation*}

By a direct computation (see for instance \cite[Exercise 3, p. 127]{St-Sh}),
we have 
\begin{equation*}
\widehat{g}_{1}(\xi )=c_{1}e^{-c|\xi |},
\end{equation*}%
and 
\begin{equation*}
\widehat{g}_{2}(\xi )=c_{2}\partial \widehat{g}_{1}(\xi )=c_{3}e^{-c|\xi
|}\sgn (\xi ),
\end{equation*}
where $c>0$, $c_{1},c_{2},c_{3}\in \mathbb{C}$ are constants. Since $%
\widehat{g}_{1},\widehat{g}_{2}\in L^{\infty }$, we have $\widehat{\rho }%
_{j}=\widehat{g}_{1}+\widehat{g}_{2}\in L^{\infty }$. \hfill $\square $

\bigskip

In the proof of our interpolation result it will be useful to have a variant
of Cauchy's formula connecting the boundary values of an analytic function
with the function itself. For this purpose we recall (a variant of) Lemma 27
from \cite[Section 3.3.1]{CE-BB}:

\begin{lemma}
\label{l1}Fix some $\alpha >0$. Suppose $\left( B_{0},B_{1}\right) $ is a
compatible couple of Banach spaces. Consider some functions\footnote{%
In \cite[Section 3.3.1]{CE-BB} we have instead the stronger condition $%
u_{j}\in C_{b}^{1}\left( \mathbb{R},B_{j}\right) \cap L^{2}(\exp (\alpha
t^{2}),B_{j}) $, for some $\alpha >0$. However, as it is easy to see from
the proof of \cite[Lemma 27]{CE-BB}, other weaker conditions are also
available.} $u_{j}\in
C_{b}^{1}\left( \mathbb{R},B_{j}\right) \cap L^{2}(\left\langle
t\right\rangle ^{d},B_{j})$, $j=0,1$, and define $u:\overline{S}\rightarrow B_{0}+B_{1}$ by 
\begin{equation}
u(z):=-\frac{1}{2\pi i}\int_{\mathbb{R}}\frac{u_{0}\left( t\right) }{it-z}dt+%
\frac{1}{2\pi i}\int_{\mathbb{R}}\frac{u_{1}\left( t\right) }{1+it-z}dt\text{%
,}  \label{Cauchy}
\end{equation}%
for all $z\in S$, and 
\begin{equation*}
u\left( j+it\right) :=H_{j}^{S}u_{j}(t)+R_{j}^{S}u_{1-j}(t)\text{,}
\end{equation*}%
for all $t\in \mathbb{R}$. Then, $u$ is analytic in $S$ and $u\in
C_{b}\left( \overline{S},B_{0}+B_{1}\right) $.
\end{lemma}

\subsection{Proof of the main results}

We use now the methods in \cite[Section 3.3.2]{CE-BB} together with the
facts from Sections \ref{sec.C} and \ref{sec.B} in order to prove a real
interpolation analogue of \cite[Lemma 29]{CE-BB}.

\begin{theorem}
\label{th.int} Let $A_{0},A_{1},B_{0},B_{1},E,F$ be some non-trivial Banach
spaces such that $A_{0},A_{1},B_{0},B_{1}\hookrightarrow E$, and consider a
bounded linear operator $T:E\rightarrow F$. Assume that $A_{0}$, $A_{1}$
have separable preduals and $E$ is separable and reflexive. Suppose
moreover, that

(i) $A_{1},B_{0}\hookrightarrow A_{0}$ and $A_{0}\cap B_{1}\hookrightarrow
A_{1}$;

(ii) $T(B_{1})\hookrightarrow T(A_{1})$.

Then, for any $\theta \in (0,1)$ and any $q\in \lbrack 1,\infty )$ we have 
\begin{equation*}
T((B_{0},B_{1})_{\theta ,q})\hookrightarrow T((A_{0},A_{1})_{\theta ,q}).
\end{equation*}
\end{theorem}

The proof of Theorem \ref{th.int} can be schematically described by the following diagram:

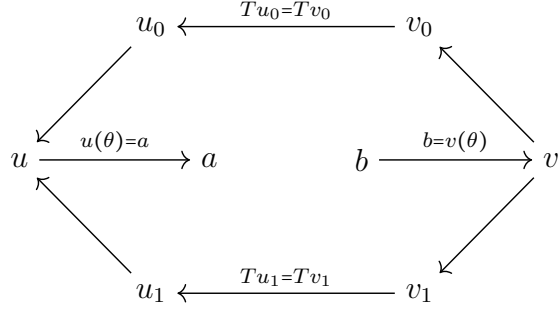
\begin{equation*}
\begin{tikzpicture}[commutative diagrams/every diagram]
\node (V0) at (45:2.5cm) {$v_{0}$};
\node (U0) at (3*45:2.5cm) {$u_{0}$} ;
\node (U1) at (5*45:2.5cm) {$u_{1}$};
\node (V1) at (7*45:2.5cm) {$v_{1}$};
\node (B) at (0:1cm) {$b$};
\node (V) at (0:3.5cm) {$v$};
\node (A) at (0:-1cm) {$a$};
\node (U) at (0:-3.5cm) {$u$};

\path[commutative diagrams/.cd, every arrow, every label]
(B) edge node {$b=v(\theta)$} (V)
(U) edge node  {$u(\theta)=a$} (A)
(V0) edge node [swap] {$Tu_{0}=Tv_{0}$} (U0)
(V1) edge node [swap] {$Tu_{1}=Tv_{1}$} (U1)
(V) edge node { } (V0)
(V) edge node { } (V1)
(U0) edge node  { } (U)
(U1) edge node  { } (U);
\end{tikzpicture}
\end{equation*}
We first pick an element $b$ from a specific dense subspace of $(B_{0},B_{1})_{\theta, q}$ and we find some function $v$ in  $\widehat{\mathcal{F}}_{00}^{q}(B_{0},B_{1})$ with $v(\theta )=b$. Then, we consider the boundary values $v_{0}$ and $v_{1}$ of $v$ on the two corresponding sides of $\partial S$, where $j=0$ or $j=1$ respectively. On each side corresponding to $j$, we find some regular function $u_{j}$ such that we have $Tu_{j}=Tv_{j}$ point-wise on $\mathbb{R}$. Next, by using Cauchy's formula (\ref{Cauchy}) we construct from $u_{0}$ and  $u_{1}$ an analytic function $u$ that belongs to $\widehat{\mathcal{F}}^{q}(A_{0},A_{1})$. One can now set $a:=u(\theta)$, obtaining by Cauchy's formula and the fact that $Tu_{j}=Tv_{j}$ the equality $Tu(\theta)=Tv(\theta)$, i.e., $Ta=Tb$. We can pass to the full space of sources $(B_{0},B_{1})_{\theta, q}$ by approximation and an argument involving the Banach-Alaoglu theorem.

\noindent \textbf{Proof of Theorem \ref{th.int}.} Consider some $b\in B_{0}\cap B_{1}$. By Lemma \ref%
{lem.d} there exists a function $v\in \widehat{\mathcal{F}}%
_{00}^{q}(B_{0},B_{1})$ with $v(\theta )=b$, such that 
\begin{equation}
\left\Vert v\right\Vert _{\widehat{\mathcal{F}}^{q}(B_{0},B_{1})}=%
\max_{j=0,1}\left\Vert \widehat{v}_{j}\right\Vert _{L^{q}(\mathbb{R}%
B_{j})}\lesssim \left\Vert b\right\Vert _{(B_{0},B_{1})_{\theta ,q}},
\label{v-b}
\end{equation}%
where $v_{j}(t):=v(j+it)$, for any $t\in \mathbb{R}$.

Since $v$ is analytic (and bounded with regular boundary values), using
Lemma \ref{l1}, we have 
\begin{equation*}
v_{j}=H_{j}^{S}v_{j}+R_{j}^{S}v_{1-j},
\end{equation*}%
and by taking the Fourier transform we obtain%
\begin{equation}
\widehat{v}_{j}=\chi _{j}\widehat{v}_{j}+\widehat{\rho }_{j}\widehat{v}%
_{1-j}.  \label{v-v}
\end{equation}

By (\ref{v-b}) and the fact that $\chi _{1}$ is a bounded function, we have $%
\widehat{v}_{1}\in L^{q}(\mathbb{R},B_{1})$ and 
\begin{equation*}
\left\Vert \chi _{1}\widehat{v}_{1}\right\Vert _{L^{q}(\mathbb{R},
B_{1})}\lesssim \left\Vert b\right\Vert _{(B_{0},B_{1})_{\theta ,q}}.
\end{equation*}

Combining this with (\ref{v-v}) for $j=1$ (and (\ref{v-b})) we get 
\begin{equation}
\left\Vert \widehat{\rho }_{1}\widehat{v}_{0}\right\Vert _{L^{q}(\mathbb{R},
B_{1})}\lesssim \left\Vert b\right\Vert _{(B_{0},B_{1})_{\theta ,q}}.
\label{ro-v-b}
\end{equation}

\textbf{I. Construction of the function }$u$\textbf{.} First we put 
\begin{equation*}
u_{0}:=v_{0}.
\end{equation*}

By (ii), for any element $\widehat{v}_{1}(\xi )\in B_{1}$, with $\xi \in 
\mathbb{R}$, there exists an element $\widehat{u}_{1}(\xi )\in A_{1}$, such
that 
\begin{equation}
T\widehat{u}_{1}(\xi )=T\widehat{v}_{1}(\xi ),  \label{T-uv}
\end{equation}%
and 
\begin{equation}
\left\Vert \widehat{u}_{1}(\xi )\right\Vert _{A_{1}}\lesssim \left\Vert 
\widehat{v}_{1}(\xi )\right\Vert _{B_{1}},  \label{uv-1}
\end{equation}%
uniformly in $\xi \in \mathbb{R}$. As in the proof of Lemma 28 in \cite[%
Section 3.3.1]{CE-BB} we can assume\footnote{This can be done as follows. We divide $\mathbb{R}$ in disjoint intervals each of length $1/N$ (where $N$ are dyadic integers) and consider the conditional  expectation $\mathbb{E}_{N}\widehat{v}_{1}$ with respect to these intervals. One can now find an appropriate simple function $\widehat{u}_{1,N}$ satisfying $T\widehat{u}_{1,N}=T\mathbb{E}_{N}\widehat{v}_{1}$ point-wise. We take now on both sides convolution with appropriate smooth functions (corresponding to the multiplication with $g_{\varepsilon}(z)=\exp(\varepsilon z^2)$, with small $\varepsilon>0$ on the side of $v$). It remains to let $N\rightarrow\infty$ and to use an Ascoli type theorem (see the Appendix of \cite{CE-BB}). We obtain $T\widehat{u}_{1}^{\varepsilon}=T(\widehat{v}_{1})_{\varepsilon}$ for some regular $\widehat{u}_{1}^{\varepsilon}$ that will be denoted by $\widehat{u}_{1}$. Here, $(\widehat{v}_{1})_{\varepsilon}$ is the Fourier transform of the boundary value of $g_{\varepsilon}v$ (on the side ``$j=1$'') which in turn approximates the initial $v$ when $\varepsilon\rightarrow 0$. Since $T$ and the multiplication by $g_{\varepsilon}$ are commuting, one can work from the beginning with $g_{\varepsilon}v$ in the place of $v$. In the proof we choose for simplicity to re-denote $g_{\varepsilon}v$ by $v$ and correspondingly $b_{\varepsilon}:=(g_{\varepsilon}v)(\theta)$ by $b$.} that $\partial ^{l}\widehat{u}_{1}\in
C_{b}^{1}\left( \mathbb{R},B_{j}\right) \cap L^{2}(\mathbb{R},B_{1})$, for
any $l\in \{0,1,...m\}$ where $m$ is a fixed positive integer (in particular
for $m=2d$). Now, we define $u_{1}:=(\widehat{u}_{1})^{\vee }$ (i.e., $u_{1}$
is the inverse Fourier transform of $\widehat{u}_{1}$) and we have $u_{1}\in
C_{b}^{1}\left( \mathbb{R},B_{1}\right) \cap L^{2}(\left\langle
t\right\rangle ^{d},B_{1})$. Note that, by definition, we also have $%
u_{0}\in C_{b}^{1}\left( \mathbb{R},B_{0}\right) \cap L^{2}(\left\langle
t\right\rangle ^{d},B_{0})$.

We construct now $u$ from $u_{0}$ and $u_{1}$, via the Cauchy's formula, as
in Lemma \ref{l1}. We have 
\begin{equation}
\widehat{u(j+i\cdot )}=\chi _{j}\widehat{u}_{j}+\widehat{\rho }_{j}\widehat{u%
}_{1-j}.  \label{u-j}
\end{equation}

\textbf{I.1. Estimate in the case where }$j=0$\textbf{.} We have $\widehat{u}_{0}=\widehat{%
v}_{0}$ and by using the boundedness of $\chi _{0}$ together with the
condition $B_{0}\hookrightarrow A_{0}$ (see (i)) and (\ref{v-b}), one can
write%
\begin{equation}
\left\Vert \chi _{0}\widehat{u}_{0}\right\Vert _{L^{q}(\mathbb{R},
A_{0})}\lesssim \left\Vert \widehat{v}_{0}\right\Vert _{L^{q}(\mathbb{R},
B_{0})}\lesssim \left\Vert b\right\Vert _{(B_{0},B_{1})_{\theta ,q}}.
\label{0-1}
\end{equation}

Also, using the boundedness of $\widehat{\rho }_{0}$ (see Lemma \ref{lem.ro}%
), together with the condition $A_{1}\hookrightarrow A_{0}$ (see (i)) and (%
\ref{v-b}), (\ref{uv-1}), we get 
\begin{eqnarray}
\left\Vert \widehat{\rho }_{0}\widehat{u}_{1}\right\Vert _{L^{q}(\mathbb{R}%
,A_{0})} &\lesssim &\left\Vert \widehat{u}_{1}\right\Vert _{L^{q}(\mathbb{R}%
,A_{0})}\lesssim \left\Vert \widehat{u}_{1}\right\Vert _{L^{q}(\mathbb{R}%
,A_{1})}  \notag \\
&\lesssim &\left\Vert \widehat{v}_{1}\right\Vert _{L^{q}(\mathbb{R}%
,B_{1})}\lesssim \left\Vert b\right\Vert _{(B_{0},B_{1})_{\theta ,q}}.
\label{0-2}
\end{eqnarray}

From (\ref{u-j}), (\ref{0-1}) and (\ref{0-2}) for $j=0$ we obtain 
\begin{equation}
\left\Vert \widehat{u(i\cdot )}\right\Vert _{L^{q}(\mathbb{R},
A_{0})}\lesssim \left\Vert b\right\Vert _{(B_{0},B_{1})_{\theta ,q}}.
\label{0-u}
\end{equation}

\textbf{I.2. Estimate in the case }$j=1$\textbf{.} By the boundedness of $\chi _{1}$
and (\ref{uv-1}), (\ref{v-b}) we have 
\begin{equation}
\left\Vert \widehat{u}_{1}\right\Vert _{L^{q}(\mathbb{R},A_{1})}\lesssim
\left\Vert \widehat{v}_{1}\right\Vert _{L^{q}(\mathbb{R},B_{1})}\lesssim
\left\Vert b\right\Vert _{(B_{0},B_{1})_{\theta ,q}}.  \label{1-1}
\end{equation}

Also, the boundedness of $\widehat{\rho }_{1}$ (see Lemma \ref{lem.ro}) and (%
\ref{ro-v-b}) give us 
\begin{equation}
\left\Vert \widehat{\rho }_{1}\widehat{u}_{0}\right\Vert _{L^{q}(\mathbb{R},
B_{1})}=\left\Vert \widehat{\rho }_{1}\widehat{v}_{0}\right\Vert _{L^{q}(%
\mathbb{R},B_{1})}\lesssim \left\Vert b\right\Vert _{(B_{0},B_{1})_{\theta
,q}}.  \label{1-2'}
\end{equation}

On the other hand, since $\widehat{\rho }_{1}$ is bounded and $%
B_{0}\hookrightarrow A_{0}$ (see (i)), 
\begin{equation}
\left\Vert \widehat{\rho }_{1}\widehat{u}_{0}\right\Vert _{L^{q}(\mathbb{R},
A_{0})}\lesssim \left\Vert \widehat{v}_{0}\right\Vert _{L^{q}(\mathbb{R},
A_{0})}\lesssim \left\Vert \widehat{v}_{0}\right\Vert _{L^{q}(\mathbb{R},
B_{0})}\lesssim \left\Vert b\right\Vert _{(B_{0},B_{1})_{\theta ,q}}
\label{1-2''}
\end{equation}%
(where we have also used (\ref{v-b})). The embedding $A_{0}\cap
B_{1}\hookrightarrow A_{1}$ (see (i)) together with (\ref{1-2'}), (\ref%
{1-2''}) yields 
\begin{equation}
\left\Vert \widehat{\rho }_{1}\widehat{u}_{0}\right\Vert _{L^{q}(\mathbb{R},
A_{1})}\lesssim \left\Vert b\right\Vert _{(B_{0},B_{1})_{\theta ,q}}.
\label{1-2}
\end{equation}

From (\ref{1-1}), (\ref{1-2}) and (\ref{u-j}) for $j=1$ we obtain 
\begin{equation}
\left\Vert \widehat{u(1+i\cdot )}\right\Vert _{L^{q}(\mathbb{R},
A_{1})}\lesssim \left\Vert b\right\Vert _{(B_{0},B_{1})_{\theta ,q}}.
\label{1-u}
\end{equation}

\textbf{II. Definition of }$a$\textbf{.} By taking the Fourier transform in (%
\ref{T-uv}) we obtain 
\begin{equation*}
Tu_{1}=Tv_{1},
\end{equation*}%
hence (since $u_{0}=v_{0}$), 
\begin{equation*}
Tu_{j}=Tv_{j},
\end{equation*}%
for any $j=0,1$. This implies via (\ref{Cauchy}) that 
\begin{equation*}
Tu(\theta )=Tv(\theta ),
\end{equation*}%
i.e., 
\begin{equation}
Ta=Tb,  \label{Ta-Tb}
\end{equation}%
where $a:=u(\theta )$. By (\ref{0-u}), (\ref{1-u}) we conclude that $a\in
(A_{0},A_{1})_{\theta ,q}^{c}=(A_{0},A_{1})_{\theta ,q}$ and the estimate 
\begin{equation}
\left\Vert a\right\Vert _{(A_{0},A_{1})_{\theta ,q}}\lesssim \left\Vert
b\right\Vert _{(B_{0},B_{1})_{\theta ,q}},  \label{Ta-Tb-1}
\end{equation}%
is satisfied.

\textbf{III. Approximation.} Now consider some $b\in (B_{0},B_{1})_{\theta ,q}$ and a sequence $%
(b_{n})_{n\geq 1}$ in $B_{0}\cap B_{1}$ such that $b_{n}\rightarrow b$ in
the norm of $(B_{0},B_{1})_{\theta ,q}$ and 
\begin{equation*}
\left\Vert b_{n}\right\Vert _{(B_{0},B_{1})_{\theta ,q}}\lesssim \left\Vert
b\right\Vert _{(B_{0},B_{1})_{\theta ,q}}.
\end{equation*}

By (\ref{Ta-Tb}), (\ref{Ta-Tb-1}) there exists a sequence $(a_{n})_{n\geq 1}$
in $(A_{0},A_{1})_{\theta ,q}$ with 
\begin{equation}
Ta_{n}=Tb_{n},  \label{Ta-n}
\end{equation}
and 
\begin{equation*}
\left\Vert a_{n}\right\Vert _{(A_{0},A_{1})_{\theta ,q}}\lesssim \left\Vert
b_{n}\right\Vert _{(B_{0},B_{1})_{\theta ,q}}\lesssim \left\Vert
b\right\Vert _{(B_{0},B_{1})_{\theta ,q}}.
\end{equation*}

Consider a decomposition of the form (see \cite[Lemma 3.2.3]{LB}) 
\begin{equation*}
a_{n}=\sum_{\nu \in \mathbb{Z}}a_{n}^{\nu },
\end{equation*}%
with 
\begin{equation}
\left( \sum_{\nu \in \mathbb{Z}}(2^{-\nu \theta }J(2^{\nu },a_{n}^{\nu
}))^{q}\right) ^{1/q}\lesssim \left\Vert a_{n}\right\Vert
_{(A_{0},A_{1})_{\theta ,q}}\lesssim \left\Vert b\right\Vert
_{(B_{0},B_{1})_{\theta ,q}}.  \label{J-an}
\end{equation}

By the Banach-Alaoglu theorem there exists a subsequence $(n_{k})_{k\geq 1}$
and elements $a^{\nu }\in A_{0}\cap A_{1}$ such that $a_{n_{k}}^{\nu
}\rightarrow a^{\nu }$, in the $w^{\ast }-$topology of $A_{0}$, $A_{1}$ and%
\footnote{%
Note that $E$ has a separable predual. This follows from the fact that if
the dual $X^{\ast }$ of a Banach space $X$ is separable, then, $X$ is
separable (see for instance \cite[Theorem 4.6-8, p. 245]{Kr}).} $E$. We also
have some $a\in E$ such that $a_{n_{k}}^{\nu }\rightarrow a^{\nu }$ in the $%
w^{\ast }-$topology of $E$ and 
\begin{equation*}
a=\sum_{\nu \in \mathbb{Z}}a^{\nu }.
\end{equation*}

For any positive integer $N$ we have by (\ref{J-an}) 
\begin{equation*}
\left( \sum_{|\nu |\leq N}(2^{-\nu \theta }J(2^{\nu },a^{\nu }))^{q}\right)
^{1/q}\lesssim \left\Vert b\right\Vert _{(B_{0},B_{1})_{\theta ,q}},
\end{equation*}%
uniformly in $N$. This implies that $a\in (A_{0},A_{1})_{\theta ,q}$ and 
\begin{equation*}
\left\Vert a\right\Vert _{(A_{0},A_{1})_{\theta ,q}}\lesssim \left\Vert
b\right\Vert _{(B_{0},B_{1})_{\theta ,q}}.
\end{equation*}

From (\ref{Ta-n}), by taking limits along $(n_{k})_{k\geq 1}$ we obtain 
\begin{equation*}
Ta=Tb,
\end{equation*}%
in $F$. This completes the proof of Theorem \ref{th.int}. \hfill $\square $

\bigskip

Theorem \ref{th.BB.g} is only a particular case of Theorem \ref{th.int}
above.

\noindent \textbf{Proof of Theorem \ref{th.BB.g}. }We apply Theorem \ref%
{th.int} to the spaces $A_{0}=B_{0}=X_{0}$, $A_{1}=X_{0}\cap X_{1}$, $%
B_{1}=X_{1}$. The condition (i) in the statement of Theorem \ref{th.int} is
trivially satisfied, while the condition (ii) follows from the assumption on
the operator $T$ in Theorem \ref{th.BB.g}. Theorem \ref{th.int} gives us the
equality 
\begin{equation*}
T((X_{0},X_{0}\cap X_{1})_{\theta ,q})=T(X_{\theta ,q}),
\end{equation*}%
and since 
\begin{equation*}
(X_{0},X_{0}\cap X_{1})_{\theta ,q}\hookrightarrow X_{0}\cap X_{\theta ,q},
\end{equation*}%
we get 
\begin{equation*}
T(X_{\theta ,q})=T((X_{0},X_{0}\cap X_{1})_{\theta ,q})\hookrightarrow
T(X_{0}\cap X_{\theta ,q})\hookrightarrow T(X_{\theta ,q}),
\end{equation*}%
which concludes the proof of Theorem \ref{th.BB.g}. \hfill $\square $

\medskip

The proof of Theorem \ref{th.BB} is now direct:

\noindent \textbf{Proof of Theorem \ref{th.BB}.} We apply Theorem \ref%
{th.BB.g}. Consider the spaces $X_{0}=L^{\infty }$, $X_{1}=X$, $E=H^{-m}$
and $F=H^{-m-1}$. Since $m>d/2$ we have $H^{m}\hookrightarrow L^{\infty
}\hookrightarrow L^{1}$ on $\mathbb{T}^{d}$, by duality we have $L^{\infty
}\hookrightarrow H^{-m}=E$. The remaining conditions are readily verified. \hfill $\square $

\begin{remark}
Note that, in contrast to Lemma 29 in \cite{CE-BB}, we do not require here the $UMD$ property of source or solution spaces. In \cite{CE-BB} this issue was related to the boundedness of the Hilbert transform. Here, however, since we work ``on the other side'' of the Fourier transform it is the point-wise multiplication by the symbol of the Hilbert transform that has to be bounded (which in our case is not a problem).
\end{remark}

\begin{remark}
The condition $q<\infty $ in Theorem \ref{th.int} or Theorem \ref{th.BB.g}
is imposed only for simplicity of the statements. For instance, in the general case $%
q\in[1,\infty] $ the statement of Theorem \ref{th.int} can be modified so that
the source space is the closure of $B_{0}\cap B_{1}$ in  $%
(B_{0},B_{1})_{\theta ,q}$.
\end{remark}

\begin{remark}
Our results are somewhat similar to the results in \cite{Mali} where the
authors obtain solutions of partial differential equations in interpolation
spaces. However, the methods used there and the differential equations
considered are of a different nature. In particular, in \cite{Mali} the
equations are determined (in contrast with the underdetermined nature of the
divergence operator considered here).
\end{remark}

\section{Further remarks}

\subsection{Remarks related to the spaces $(L^{\infty },H^{d/2})_{\theta ,q}$}

\label{s.f.remarks2}

We note first that Corollary \ref{cor.BB} is not trivial, in the sense that
we have the non-embedding (\ref{nontriv}). This immediately follows from the
fact that $H^{d/2}\not\hookrightarrow L^{\infty }$ via the $J-$method of
interpolation. One can construct examples as follows. For each positive
integer $\nu $ there exists $u_{\nu }\in C_{c}^{\infty }(B(0,1))$ with $%
u_{\nu }(0)=1$, such that 
\begin{equation*}
\left\Vert u_{\nu }\right\Vert _{L^{\infty }}\leq 1\text{ \ \ and \ \ }%
\left\Vert u_{\nu }\right\Vert _{H^{d/2}}\leq 2^{-\nu }\text{.}
\end{equation*}

If we consider now the sequence $(f_{n})_{n\geq 1}$, defined by 
\begin{equation*}
f_{n}:=\sum_{\nu =1}^{n}u_{\nu },
\end{equation*}%
we get that (since $J(2^{\nu },u_{\nu })\lesssim 1$) 
\begin{equation*}
\left\Vert f_{n}\right\Vert _{(L^{\infty },H^{d/2})_{\theta ,q}}\leq \left(
\sum_{\nu =1}^{n}(2^{-\nu \theta }J(2^{\nu },u_{\nu }))^{q}\right)
^{1/q}\lesssim 1,
\end{equation*}%
uniformly in $n$ (see \cite[Lemma 3.2.3]{LB}), while $f_{n}(0)\geq n$.

\bigskip Now let us discuss one potential improvement of Corollary \ref%
{cor.BB}. Suppose for simplicity that $1/q=\theta /2$. We have the embedding 
\begin{equation}
X_{\theta ,q}=(L^{\infty },H^{d/2})_{\theta ,q}\hookrightarrow
(BMO,H^{d/2})_{\theta ,q}=B_{q}^{\sigma ,q},  \label{eq-emb}
\end{equation}%
where $\sigma =\theta d/2$. As we have mentioned, we do not have an explicit
easy description of the space $X_{\theta ,q}$. However, instead of computing
the space $X_{\theta ,q}$, one may look to its divergence. It is reasonable
to believe that 
\begin{equation}
\div  (L^{\infty },H^{d/2})_{\theta ,q}=B_{q,\sharp }^{\sigma -1,q}.
\label{eq-div}
\end{equation}%
\bigskip

By duality, this is equivalent to 
\begin{equation}
(G(L^{1}),G(H^{-d/2}))_{\theta ,q^{\prime }}=G((L^{1},H^{-d/2})_{\theta
,q^{\prime }}),  \label{eq-grad}
\end{equation}%
where, for any function space $Y$ on $\mathbb{T}^{d}$, $G(Y)$ is the
subspace of gradients in $Y$:

\begin{equation*}
G(Y):=\{\nabla f\mid f\in \mathcal{D}^{\prime }(\mathbb{T}^{d})\text{, }%
\nabla f\in Y\}\subset Y^{d},
\end{equation*}%
endowed with the norm induced by $Y$.

Indeed, if (\ref{eq-grad}) holds, then by duality we get%
\begin{equation*}
((W^{1,1},H^{-d/2+1})_{\theta ,q^{\prime }}^{\ast })_{\sharp }=\div 
(L^{\infty },H^{d/2})_{\theta ,q},
\end{equation*}%
and in order to get (\ref{eq-div}) it suffices to observe that 
\begin{equation*}
(W^{1,1},H^{-d/2+1})_{\theta ,q^{\prime }}=B_{q^{\prime },\sharp }^{-\sigma
+1,q^{\prime }},
\end{equation*}%
which follows from the results of \cite{CDDV} (see also \cite[Proposition 1.3%
]{CE-1}).

In turn, (\ref{eq-grad}) would be implied by the the $K-$closedness of the
subcouple $(G(L^{1}),G(H^{-d/2}))$ in $(L^{1},H^{-d/2})$. In other words,
one may ask whether or not 
\begin{equation}
K_{t}(f,G(L^{1}),G(H^{-d/2}))\sim K_{t}(f,L^{1},H^{-d/2}),  \label{eq-K}
\end{equation}%
for any $f\in G(L^{1})+G(H^{-d/2})$ and any $t>0$, where, as usual, 
\begin{equation*}
K_{t}(f,Y_{0},Y_{1}):=\inf_{f=f_{0}+f_{1}}(\Vert f_{0}\Vert _{Y_{0}}+t\Vert
f_{1}\Vert _{Y_{1}}),
\end{equation*}%
for any interpolation couple $(Y_{0},Y_{1})$ of function spaces.

There are several results in the literature that concern the $K-$closedness
of similar subcouples. For instance, by a result of Bourgain (\cite[Theorem 3%
]{B}) we have that $(G(L^{1}),G(L^{p}))$ is $K-$closed in $(L^{1},L^{p})$,
for any $p\in (1,\infty )$. More generally, using the methods in \cite{KK}
we have that $(G(L^{1}),G(W^{l,p}))$ is $K-$closed in $(L^{1},W^{l,p})$, for
any $p\in (1,\infty )$ and any non-negative integer $l$. Also, one can
observe that Mazya's result can be formulated as the fact that (\ref{eq-K})
holds for any $0<t\leq 1$.

\bigskip

In the case where (\ref{eq-K}) holds (or (\ref{eq-div}) holds directly),
then, by (\ref{eq-div}) and Corollary \ref{cor.BB} we would get an
improvement of Corollary \ref{cor.BB}:

\begin{corollary}
\label{cor.BB'}Suppose (\ref{eq-div}) holds. For any $\theta \in (0,1)$, and 
$q$ with $1/q=\theta /2$ we have 
\begin{equation*}
\div  (L^{\infty }\cap (L^{\infty },H^{d/2})_{\theta ,q})=B_{q,\sharp
}^{\sigma -1,q}=\div  B_{q}^{\sigma ,q},
\end{equation*}%
where $\sigma =\theta d/2$. In particular, by (\ref{eq-emb}), $B_{q}^{\sigma
,q}$ is a $BB$ space.
\end{corollary}

\subsection{More exotic $BB$ spaces}

Let us start by introducing a new space. We denote by $S_{1}L^{\infty }$ the
space of those distributions $f$ on the torus for which the norm 
\begin{equation*}
\left\Vert f\right\Vert _{S_{1}L^{\infty }}:=\left\Vert \sum_{k\geq
0}\left\vert P_{k}f\right\vert \right\Vert _{L^{\infty }},
\end{equation*}%
is finite, where $P_{k}$ are the usual Littlewood-Paley projections. 

Consider also the space the space $\tilde{F}_{\infty }^{0,1}$ defined as the
completion of $C^{\infty }$ with respect to the norm%
\begin{equation*}
\left\Vert f\right\Vert _{\tilde{F}_{\infty }^{0,1}}:=\inf \left\{ \left.
\left\Vert\sup_{k\geq 0} |f_{k}|\right\Vert _{L^{1}}\right\vert \text{ }%
f=\sum_{k\geq 0}P_{k}f_{k}, \text{ with } (\left\Vert f_{k}\right\Vert _{L^{\infty}})_{k\geq 0}\in c_{0}\right\}.
\end{equation*}

One can check that $\tilde{F}_{\infty }^{0,1}$ is separable and $(\tilde{F}%
_{\infty }^{0,1})^{\ast }=S_{1}L^{\infty }$. The spaces $S_{1}L^{\infty }$
and $\tilde{F}_{\infty }^{0,1}$ are not Triebel-Lizorkin spaces (in
particular, $\tilde{F}_{\infty }^{0,1}$ is not the Triebel-Lizorkin space $%
F_{\infty }^{0,1}$; see for instance the discussion in \cite[Section 3.5.2]%
{S-T}). Note that we trivially have $S_{1}L^{\infty }\hookrightarrow
L^{\infty }$. Hence, the following is an improvement of Mazya's result:

\begin{proposition}
\label{prop.SL}We have 
\begin{equation*}
\div  (S_{1}L^{\infty }\cap H^{d/2})=\div  H^{d/2}=H_{\sharp
}^{d/2-1}.
\end{equation*}
\end{proposition}

The proof of Proposition \ref{prop.SL} follows closely the proof of Lemma 37
in \cite{CE-BB}. Let us mention here the main ingredient.

We need first to adapt the definition of the $\ell $-$BB$ symbols to the
discrete setting that parallels the definition in \cite{CE-BB}. Let $1\leq
\ell \leq d$ be some integers and let $m:\mathbb{Z}^{d}\rightarrow \mathbb{C}
$ be a function. We say that $m$ is an $\ell $-$BB$ symbol on $\mathbb{Z}%
^{d} $ if the following conditions are satisfied:

\begin{itemize}
\item[(i)] there exists a constant $C>0$ such that, in the case $\ell <d$,%
\begin{equation}
\sum_{n^{\prime \prime }\in \mathbb{Z}^{d-\ell }}\left\vert \partial
_{1}^{\alpha _{1}}...\partial _{\ell }^{\alpha _{l}}m(n^{\prime },n^{\prime
\prime })\right\vert \leq \frac{C}{\left\vert n^{\prime }\right\vert ^{\ell
+\left\vert \alpha \right\vert }}\text{, \ }  \label{BB1-tor}
\end{equation}
for all $\alpha =\left( \alpha _{1},...,\alpha _{\ell }\right) \in \left\{
0,1\right\} ^{\ell }$ \ and all $n^{\prime }\in \mathbb{Z}_{+}^{\ell }$,
and, in the case $\ell =d$, 
\begin{equation}
\left\vert \partial _{1}^{\alpha _{1}}...\partial _{d}^{\alpha
_{d}}m(n)\right\vert \leq \frac{C}{\left\vert n\right\vert ^{\ell
+\left\vert \alpha \right\vert }}\text{, \ }  \label{BB2-tor}
\end{equation}%
for all $\alpha =\left( \alpha _{1},...,\alpha _{d}\right) \in \left\{
0,1\right\} ^{d}$ \ and all $n\in \mathbb{Z}_{+}^{d}$;

\item[(ii)] $m$ is an odd function in each of the components $n_{1}$, $n_{2}$%
,....,$n_{\ell }$, i.e.,

\begin{equation*}
m\left( n_{1},...,n_{j-1},-n_{j},n_{j+1}....,n_{d}\right) =-m\left(
n_{1},...,n_{j-1},n_{j},n_{j+1}....,n_{d}\right) \text{, \ }
\end{equation*}%
for all $1\leq j\leq \ell $, and all $n_{1},n_{2},....,n_{d}\in \mathbb{Z}$.
\end{itemize}

In (\ref{BB1-tor}) and (\ref{BB2-tor}) the partial derivative $\partial _{j}$
has the following meaning. If $a:\mathbb{Z}^{d}\rightarrow \mathbb{C}$ is a
function, then, $\partial _{j}a:\mathbb{Z}^{d}\rightarrow \mathbb{C}$ is the
function defined by:%
\begin{equation*}
\partial _{j}a(n_{1},...,n_{d})=a\left(
n_{1},...,n_{j-1},n_{j}+1,n_{j+1}....,n_{d}\right) -a\left(
n_{1},...,n_{j-1},n_{j},n_{j+1}....,n_{d}\right) ,
\end{equation*}%
for all $n_{1},n_{2},....,n_{d}\in \mathbb{Z}$.

\bigskip

With minor modifications in the proof of \cite[Lemma 28]{CE-BB} one can
obtain the following:

\begin{lemma}
\label{Rmult}Let $d\geq 1$ be an integer and let $m:\mathbb{R}%
^{d}\rightarrow \mathbb{C}$ be an $\ell $-$BB$ symbol for some $1\leq \ell
\leq d$ and some constant $C$. Then, $K\in S_{1}L^{\infty }$ and $\left\Vert
K\right\Vert _{S_{1}L^{\infty }}\lesssim C$, where $K$ is the inverse
Fourier transform of $m$ (i.e., $\widehat{K}:=m$).
\end{lemma}

\begin{remark}
In fact, what we obtain by following the proof of \cite[Lemma 28]{CE-BB} is
a much more general statement involving the control of Fourier projections
on symmetrizations of dyadic boxes. For the sake of simplicity however, we
only mention here Lemma \ref{Rmult} which is a direct consequence of the
torus version of \cite[Lemma 28]{CE-BB}. (Also, using the torus analogue of
Lemma 37 in \cite{CE-BB} one can obtain versions of Corollary \ref{cor.BB}
for more divergence-like equations as in \cite{CE-BB}.)
\end{remark}

Proposition \ref{prop.SL} can be obtained by following closely the proof of
Lemma 37 in \cite{CE-BB}, with some minor changes: we use Lemma \ref%
{Rmult} above instead of Lemma 37 in \cite{CE-BB}; we use the standard
derivatives ($\partial _{j}$ instead of the multipliers $\partial
_{j}^{\sigma }$); instead of using the space $Y^{\ast }/U$ we use $H^{-d/2}$%
, and instead of $(\mathcal{M}+Y^{\ast })/U$ \ we use \ \ $\tilde{F}_{\infty
}^{0,1}+H^{-d/2}$.

Combining Proposition \ref{prop.SL} and Theorem \ref{th.BB.g} for $%
X_{0}:=S_{1}L^{\infty }$ and $X_{1}:=H^{d/2}$ we get

\begin{corollary}
For any $\theta \in (0,1)$, $q\in \lbrack 1,\infty )$ we have 
\begin{equation*}
\div  (S_{1}L^{\infty }\cap X_{\theta ,q})=\div  X_{\theta ,q},
\end{equation*}%
where $X_{\theta ,q}:=(S_{1}L^{\infty },H^{d/2})_{\theta ,q}$. In
particular, $X_{\theta ,q}$ is a $BB$ space.
\end{corollary}

\end{document}